\RequirePackage{amsmath}
\documentclass[a4paper, 10pt, reqno]{amsart}
\usepackage{amssymb,url, color, mathrsfs}
\usepackage[colorlinks=true, bookmarks=true, pdfstartview=FitH, pagebackref=true, linktocpage=true, linkcolor = magenta, citecolor = blue]{hyperref}
\usepackage[short,nodayofweek]{datetime}

\theoremstyle{plain}
\numberwithin{equation}{section}
\newtheorem{theorem}{Theorem}[section]
\newtheorem{lemma}[theorem]{Lemma}
\newtheorem{corollary}[theorem]{Corollary}

\theoremstyle{remark}

\hypersetup{
pdftitle={A pointwise inequality for a biharmonic equation with negative exponent and related problems},
pdfauthor={Quoc Anh Ngo, Van Hoang Nguyen, Quoc Hung Phan},
colorlinks = true,
linkcolor = magenta,
citecolor = blue,
}

\numberwithin{equation}{section}

\title[A pointwise inequality for a biharmonic equation ]{A pointwise inequality for a biharmonic equation with negative exponent and related problems}

\def\cfac#1{\ifmmode\setbox7\hbox{$\accent"5E#1$}\else\setbox7\hbox{\accent"5E#1}\penalty 10000\relax\fi\raise 1\ht7\hbox{\lower1.05ex\hbox to 1\wd7{\hss\accent"13\hss}}\penalty 10000\hskip-1\wd7\penalty 10000\box7 }

\author[Q. A. Ng\^o]{Qu\cfac oc Anh Ng\^o}
\address[Q.A. Ng\^o]{Department of Mathematics\\
College of Science, Vi\^{e}t Nam National University\\
H\`{a} N\^{o}i, Vi\^{e}t Nam.}
\email{\href{mailto: Q.A. Ng\^o <nqanh@vnu.edu.vn>}{nqanh@vnu.edu.vn}}
\email{\href{mailto: Q.A. Ng\^o <bookworm\_vn@yahoo.com>}{bookworm\_vn@yahoo.com}}

\author[V. H. Nguyen]{Van Hoang Nguyen}
\address[V.H. Nguyen]{Institut de Math\'ematiques de Toulouse\\
Universit\'e Paul Sabatier\\
31062 Toulouse c\'edex 09, France.}
\email{\href{mailto: V.H. Nguyen <van-hoang.nguyen@math.univ-toulouse.fr>}{van-hoang.nguyen@math.univ-toulouse.fr}}

\author[Q. H. Phan]{Quoc Hung Phan$^{\ast}$}
\address[Q. H. Phan]{Institute of Research and Development\\
Duy Tan University\\
Da NangVietnam.}
\email{\href{mailto: P.Q. Hung <phanquochung@dtu.edu.vn>}{phanquochung@dtu.edu.vn}}

\begin{document}

\allowdisplaybreaks

\begin{abstract}
	Inspired by a recent pointwise differential inequality for positive bounded solutions of the fourth-order H\'enon equation $\Delta^2 u = |x|^a u^p$ in ${\mathbb R}^n$ with $a \geqslant 0$, $p > 1$, $n \geqslant 5$ due to Fazly, Wei, and Xu [\textit{Anal. PDE} \textbf{8} (2015) 1541--1563], first for some positive constants $\alpha$ and $\beta$ we establish the following pointwise inequality
	\[
	\Delta u \geqslant \alpha u^{-\frac{q-1}2} + \beta u^{-1} |\nabla u|^2
	\]
in ${\mathbb R}^n$ with $n \geqslant 3$ for positive $C^4$-solutions of the fourth-order equation
	\[
	\Delta^2u=-u^{-q} \quad \text{ in } \mathbb R^n
	\]
where $q > 1$. Next, we prove a comparison property for Lane--Emden  system with exponents of mixed sign. Finally, we give an analogue result for parabolic models by establishing a comparison property for parabolic system of Lane--Emden type. To obtain all these results, a new argument of maximum principle is introduced, which allows us to deal with solutions with high growth at infinity. We expect to see more applications of this new method to other problems in different contexts.
\end{abstract}

\date{\bf \today \, at \currenttime}

\maketitle


\let\thefootnote\relax\footnote{$^{\ast}$ Corresponding author.}

\section{Introduction}
Pointwise estimates for solutions of partial differential equations (PDEs) have had tremendous impact on the existing theory of nonlinear PDEs. Various celebrated pointwise estimates for certain elliptic equations and systems have profound applications to tackle important problems in the research fields. For interested readers, we refer to the introduction of \cite{FWX15} where several concrete examples of pointwise estimates as well as their impact are mentioned.

The motivation of writing this paper goes back to the work of Fazly, Wei, and Xu \cite{FWX15} in which an important pointwise differential inequality for positive bounded solutions of the fourth-order H\'enon equation
\begin{equation}\label{positive}
\Delta^2 u = |x|^a u^p
\end{equation}
in ${\mathbb R}^n$ with $a \geqslant 0$, $p>1$, and $n \geqslant 5$ was found. To understand the surprising pointwise differential inequality in \cite{FWX15} and its role, let us consider the special case $a = 0$. In this scenario, the equation \eqref{positive} can be rewritten as the system
\begin{equation}\label{eqSystemPositiveExponent}
\left\{
\begin{split}
-\Delta u &=v,\\
-\Delta v &= u^p,\\
\end{split}
\right.
\end{equation}
in ${\mathbb R}^n$. Apparently, this system is closely related to the famous Lane--Emden system 
\begin{equation}\label{LEsystem}
\left\{
\begin{split}
-\Delta u &=v^r,\\
-\Delta v &= u^p\\
\end{split}
\right.
\end{equation}
in ${\mathbb R}^n$ where, without loss of generality, we may assume $p\geqslant r>0$. A very interesting question concerning the system~\eqref{LEsystem} is \textit{the Lane--Emden conjecture}, which states that system~\eqref{LEsystem} has no entire positive solution if and only if 
\begin{equation}\label{hyperbola}
\frac{1}{p+1}+\frac{1}{r+1}>1-\frac{2}{n}. 
\end{equation}
Toward tackling the Lane--Emden conjecture, a pointwise differential inequality for positive solutions of \eqref{eqSystemPositiveExponent} was found by Souplet. To be more precise, it was proved in \cite[Lemma 2.7]{Sou09} that
\begin{equation}\label{souplet_ine}
\frac{u^{p+1}}{p+1}\leqslant \frac{v^{r+1}}{r+1}.
\end{equation}
This inequality was then extended to positive solutions of the H\'enon--Lane--Emden system in \cite{Pha12}. Such a comparison as in \eqref{souplet_ine} is highly important as it allows us to obtain various Liouville-type results for stable solutions, see e.g., \cite{WY13, WXY13, FG13, Cow13}. We would like to mention that this kind of comparison property was proved for Dirichlet problem in bounded domain by Bidaut-V\'eron in \cite[Remark 3.2]{BV99}. 

Clearly, the inequality \eqref{souplet_ine} in the particular case $r=1$ provides a pointwise inequality for solutions of the equation~\eqref{eqSystemPositiveExponent} with $a=0$ as 
\begin{equation}\label{eq:WeakPointwiseInequalityPositive}
- \Delta u \geqslant \sqrt{\frac 2{p+1}} u^{\frac{p+1}2}
\end{equation}
in ${\mathbb R}^n$. Motivated by the pointwise inequality \eqref{eq:WeakPointwiseInequalityPositive} for solutions of \eqref{positive} with $a=0$, the main result of \cite{FWX15} is to improve this inequality by slightly increasing the coefficient $\sqrt{2/(p+1)}$ and adding a positive term involving gradient $u^{-1}|\nabla u|^2$ to the right hand side. To be exact, the inequality
\begin{equation}\label{eqFWXInequality}
- \Delta u \geqslant \sqrt{\frac 2{p + 1 - c_n}} u^{\frac{p+1}2} + \frac 2{n-4} \frac{|\nabla u|^2}{u}
\end{equation}
was proved, where $c_n = 8/(n(n-4))$.

In the present paper, we are interested in a counterpart of \cite{FWX15} by considering pointwise inequalities for positive $C^4$-solutions of the elliptic equation
\begin{equation}\label{eqMain}
\Delta^2 u= - u^{-q}
\end{equation}
in ${\mathbb R}^n$ with a negative exponent $-q <-1$. To understand why we work on such an equation, let us mention that equations of the form \eqref{positive} are essentially the projection of the fourth-order $Q$-curvature equation on $\mathbb S^n$ onto ${\mathbb R}^n$ with $n \geqslant 5$ via the stereographic projection. Under the dimension constraint $n \geqslant 5$, the resulting equations have positive exponents. However, since the $Q$-curvature equations can be posed on $\mathbb S^3$, also via the stereographic projection, the projected equations on $\mathbb S^3$ now have negative exponents of the form \eqref{eqMain}, see \cite{CX09}. However, in order to provide a good tool for analysis, Eq. \eqref{eqMain} can also be studied in ${\mathbb R}^n$ with arbitrary $n \geqslant 3$. For interested readers, we refer to \cite{KR03, Xu05, GW08, CX09, DFG10, Gue12, GW14, DN16} for various studies related to Eq. \eqref{eqMain}.

Inspired by the pointwise inequality \eqref{eqFWXInequality} for solutions of \eqref{positive}, the first purpose of this paper is to establish certain pointwise inequality for positive $C^4$-solutions of the equation \eqref{eqMain}. To seek for a candidate, we observe that a weak inequality in a same fashion of \eqref{eq:WeakPointwiseInequalityPositive} was already found by Guo and Wei in \cite[Proposition 2.5]{GW14}. The weak inequality in \cite{GW14} can be stated as follows: Let $u>0$ be a $C^4$-solution to \eqref{eqMain} in ${\mathbb R}^n$ with $q > 1$. Then the following inequality holds
\begin{equation}\label{eq:WeakPointwiseInequality}
\Delta u \geqslant \sqrt{\frac 2{q-1}} u^{-\frac{q-1}2} 
\end{equation}
in ${\mathbb R}^n$ with $n \geqslant 3$. Inspired by the inequality \eqref{eqFWXInequality} for positive bounded solutions of \eqref{positive} and the inequality \eqref{eq:WeakPointwiseInequality} for positive solutions of \eqref{eqMain}, the first task of this paper is to prove the following result.

\begin{theorem}\label{thmPointwise}
Let $u>0$ be a $C^4$-solution to \eqref{eqMain} in ${\mathbb R}^n$ with $n \geqslant 3$  and $q>1$. Assume that two positive constants $\alpha$, $\beta$ satisfy 
\begin{equation}\label{conditionI}
\left\{
\begin{split}
\alpha &\leqslant \frac{1}{2}, \\
\beta &\leqslant \sqrt{\frac{2}{q-1-4\alpha/n}}, \\
q &\geqslant 3\alpha+\sqrt{9\alpha^2+(1-2\alpha)(1+16\alpha/n)}.
\end{split}
\right.
\end{equation}
Then the following pointwise inequality 
\begin{equation}\label{eqPointwiseInequality}
\Delta u \geqslant \alpha u^{-1} |\nabla u|^2 +\beta u^{-\frac{q-1}2} 
\end{equation}
holds in ${\mathbb R}^n$ for any solution $u$ satisfying the growth condition
\begin{equation}\label{hypo}
\tag{$\mathcal H$}
u(x)=o \big(|x|^{2/(1-\gamma)} \big)\quad  \text{ as } |x|\to \infty,
\end{equation} 
where $\gamma$ is arbitrary in $[0,1)$ such that
\begin{align}\label{ConditionsForH2}
\left\{
\begin{aligned}
\alpha+\frac{4\alpha(1-2\alpha)}{n} -3\alpha\gamma-\gamma^2+\gamma>&0,\\
q-1-\frac{8\alpha}{n}-2\gamma >&0.
\end{aligned}
\right.
\end{align} 

In particular, the pointwise inequality \eqref{eqPointwiseInequality} always holds under the  assumption \eqref{conditionI}  and 
\begin{equation} \label{O2}
u(x)=O(|x|^2) \quad \text{ as } |x|\to \infty.
\end{equation}
\end{theorem}

We note that the assumptions  $n \geqslant 3$  and $q>1$ in Theorem~\ref{thmPointwise} are necessary, since there is no positive classical solution when $n=1,2$, or $n\geq 3$ and $0<q\leq 1$, see \cite[Theorem 1.3 and Remark 4.4]{LY16}. Also, the assumption \eqref{hypo} (or \eqref{O2} in particular)  is not too strict since, it was shown in \cite[Corollary 4.3]{LY16} that, all radial solutions of Eq. \eqref{eqMain} in ${\mathbb R}^n$ grow at most quadratically at infinity. 

The proof of Theorem~\ref{thmPointwise} is a refinement of the technique of Fazly, Wei, and Xu used in \cite{FWX15}, which is based on a Moser iteration-type argument, maximum principle, and a feedback argument, see \cite{Sou09}. However, unlike the case of equations with positive exponents, significant new ideas have to  be introduced due to the difficulties arising in the case of negative exponents. For instance, the technique in \cite{FWX15} requires a sufficiently decay of solutions at infinity in order to control the integrals on the sphere $|x|=R$ as $R\to \infty$. Hence, the boundedness assumption of solutions is initially assumed in \cite{FWX15} to get more decay estimates. Such a boundedness assumption is reasonable since all radial positive solutions of \eqref{positive} with $a=0$ is of growth $|x|^{-4/(p-1)}$ as $|x|\to \infty$, for example see \cite{GG06}. 

As noticed above, the situation is more challenging in case of \eqref{eqMain} with negative exponent since there do exist solutions which grow linearly or super-linearly at infinity, see e.g., \cite{CX09, KR03, Gue12, DN16}. Motivated by recent paper of Cheng, Huang, and Li \cite{CHL15}, we develop a new argument of maximum principle, which allows to deal with all solutions with high growth (or even oscillation) at infinity, see assumption \eqref{hypo}. Surprisingly, it is worth noticing that our new argument of maximum principle can be applied to the parabolic system of Lane-Emden type,  see Theorem~\ref{th16} below.
 
Concerning the pointwise inequality for problem~\eqref{eqMain}, the Moser-type iteration argument introduced in \cite{FWX15} encounters some restriction on the control of $I^{(3)}_{\varepsilon, \alpha_k,\beta_k}$ given in \cite[page 1551]{FWX15}, which leads to an additional condition that $q$ must be sufficiently large. To avoid any restriction on $q$, we introduce a new approach without using iteration argument to establish the pointwise inequality which is capable of handling any $q>1$. When looking back at \eqref{conditionI}, it is likely to restrict $q$. However, a careful examination shows that
\[
q_\alpha:=3\alpha+\sqrt{9\alpha^2+(1-2\alpha)(1+16\alpha/n)}\to 1
\]
as $\alpha\searrow 0$ and
\[
q_\alpha:=3\alpha+\sqrt{9\alpha^2+(1-2\alpha)(1+16\alpha/n)} \to 3
\]
as $\alpha\nearrow 1/2$. In other words, our pointwise estimate \eqref{eqPointwiseInequality} can handle any $q>1$; however, as a price we pay, the closer $q$ to $1$ we want, the smaller $\alpha$ we must assume.

As a direct consequence of Theorem~\ref{thmPointwise}, by taking $\alpha=1/2$, we obtain the following pointwise inequality:
\begin{corollary}\label{cor1}
Let $u>0$ be a $C^4$-solution to \eqref{eqMain} in ${\mathbb R}^n$ with $n \geqslant 3$ which satisfies the growth assumption
\begin{equation}\label{hypo6}
u(x)=o(|x|^{\tau})
\end{equation}  
as $|x|\to \infty$, with 
$$\tau:= \min\bigg\{4,\; \frac{4}{(-q+3+4/n)_{+}}\bigg\}.$$
Then the following pointwise inequality 
\begin{equation}\label{cor2ine}
\Delta u \geqslant \frac{1}{2} \frac{ |\nabla u|^2}u +\sqrt{\frac{2}{q-1-2/n}} u^{-\frac{q-1}2} 
\end{equation}
holds in ${\mathbb R}^n$ for all $q\geqslant 3$.
\end{corollary}
In Corollary \ref{cor1} above, it is worth noticing that 
$$  \min\bigg\{4,\; \frac{4}{(-q+3+4/n)_{+}}\bigg\}\geqslant  \min\big\{4,\; n\big\}\geqslant 3$$
for any $q\geqslant 3$. In addition, as an immediate consequence of \eqref{cor2ine}, we can show that under the conformal change $g = u^{2/(n-2)} g_{{\mathbb R}^n}$, this new metric has strictly negative scalar curvature everywhere in ${\mathbb R}^n$. This is because
\[
{\rm scal}_g = - \frac{2(n-1)}{n-2} \Big( \Delta u - \frac 12 \frac{|\nabla u|^2}{u} \Big) u^{-\frac{n}{n-2}}.
\]

In Theorem~\ref{thmPointwise}, we restrict ourselves to the case $\alpha, \beta>0$ for simplicity of the presentation. However, by an argument totally similar to the one used in the proof of Theorem~\ref{thmPointwise}, one can separately consider the case $\alpha=0$ or $\beta=0$. In the case $\alpha=0$, via this technique the largest $\beta$ is $\sqrt{2/(q-1})$, leading us to the pointwise inequality \eqref{eq:WeakPointwiseInequality} without assumption \eqref{hypo}. In the case $\beta=0$,  noting that the conditions $I_3\geqslant 0$ and $K_2>0$ can be relaxed when considering the inequality \eqref{keyinequality}, we then get the following result: 

\begin{corollary}\label{cor3}
	Let $u>0$ be a $C^4$-solution to \eqref{eqMain} in ${\mathbb R}^n$ with $n \geqslant 3$ which satisfies the growth assumption
	\begin{equation}\label{hypo7}
	u(x)=o(|x|^{4})
	\end{equation}  
as $|x|\to \infty$. Then the following pointwise inequality 
	\begin{equation}\label{cor3ine}
	\Delta u \geqslant \frac{1}{2} \frac{ |\nabla u|^2}u
	\end{equation}
	holds in ${\mathbb R}^n$ for all $q>1$.
\end{corollary}

We see that the growth assumption on solutions is not required when $\alpha=0$. Heuristically, this is due to the fact that the term $u^{-(q-1)/2}$ looks small in comparison with the term $\Delta u$ when $u$ grows fast. However, it is an open question whether the growth assumption \eqref{hypo} is necessary in the case $\alpha>0$.  From the weak inequality \eqref{eq:WeakPointwiseInequality}, even when $\beta$ is small, it seems difficult to obtain \eqref{eqPointwiseInequality} for some $\alpha>0$ without assuming any restriction on the growth of solutions at infinity. This is supported by the fact that the term $u^{-(q-1)/2}$ is not large enough when $u$ grows fast. 

The second topic concerns the comparison property of the Lane--Emden system with exponents of mixed sign. Let us consider the positive solutions of the following system 
\begin{equation}\label{system}
\begin{cases}
 \Delta u = v^r,\\
\Delta v= -u^{-q},
\end{cases}
\end{equation}
 in ${\mathbb R}^n$, where $q>1$ and $r>0$.   Due to the super polyharmonic property in \cite[Lemma 4.1]{LY16}, Eq. \eqref{eqMain} can be considered as a special case of system \eqref{system} when $r=1$. Our second result is the following:
\begin{theorem}\label{th2}
Let $(u,v)$ be a positive solution of system~\eqref{system}  in $\mathbb R^n$ with $q>1$ and $r>0$. Then we have the  following pointwise inequality
\begin{equation}\label{eqMain1}
\frac{v^{r+1}}{r+1} \geqslant \frac{u^{-q+1}}{q-1}
\end{equation}
in ${\mathbb R}^n$.
\end{theorem}
 
We stress that the pointwise inequality \eqref{eqMain1} holds for any $r>0$ and $q>1$. Moreover, no additional assumption on the growth of solutions is required.

The proofs of Theorems \ref{thmPointwise} and \ref{th2} follow the standard way which consists of two steps. The first step is to derive an appropriate differential inequality for an auxiliary function, then in the second step, we apply the maximum principle to prove that the auxiliary function has a sign. In a standard way, the maximum principle can only be applied when solutions have enough decay at infinity. However, we show that the decay of solutions is not necessary in the proof of pointwise inequality in some particular cases. 

In the last part of this paper, we would like to highlight our approach used in this paper by proving a poinwise inequality for parabolic problems. We provide here an application by considering the parabolic system of Lane--Emden type
\begin{equation}\label{parasystem}
\begin{cases}
u_t-\Delta u = v^r,\\
\,v_t-\Delta v= u^{p},
\end{cases}
\end{equation}
in ${\mathbb R}^n\times {\mathbb R} $, where $p, r>0$ and $pr>1$. Our result is as follows:
\begin{theorem}\label{th16}
	Let $(u,v)$ be a positive solution of system~\eqref{parasystem} in ${\mathbb R}^n\times {\mathbb R} $. Assume that $p\geqslant r>0$ and $pr>1$, then we have the  following pointwise inequality
	\begin{equation}\label{paraine}
	\frac{v^{r+1}}{r+1} \geqslant \frac{u^{p+1}}{p+1}
	\end{equation}
in ${\mathbb R}^n\times {\mathbb R} $.
\end{theorem}

We note that the Liouville-type theorem for system~\eqref{parasystem} is conjectured to hold  with $(p,r)$ in the  range \eqref{hyperbola} and $pr>1$. However, it is still an  open question, even in dimension $n=1$ or in the class of radial solutions, we refer to papers \cite{EH91, AHV97,  Zaa01} for some partial results. It is expected that the pointwise inequality \eqref{paraine} is an important step to tackle the Liouville-type for the system~\eqref{parasystem}. 
 
The rest of the paper is organized as follows. Sections 2, 3, and 4 are devoted to the proofs of Theorems \ref{thmPointwise}, \ref{th2}, and \ref{th16}, respectively.

\tableofcontents
 

\section{Pointwise inequality for the biharmonic equation (\ref{eqMain}): Proof of Theorem~\ref{thmPointwise}}
\label{sec-DevelopingIteration}

For simplicity, we denote by $p = (q-1)/2$, $A=u^{-1}|\nabla u|^2$, and $B=u^{-p}$. We define the functions $w$ by 
\begin{equation}\label{def:w}
w = -\Delta u + \alpha A +\beta B,
\end{equation}
where $\alpha$ and $\beta$ are positive constants to be chosen later. The key step in the proof of Theorem~\ref{thmPointwise} is the following lemma.

\begin{lemma}\label{lem1}
Let $u$ be a smooth positive solution of the equation \eqref{eqMain} in ${\mathbb R}^n$. Then the function $w$ defined in \eqref{def:w} satisfies in ${\mathbb R}^n$ the differential inequality
\begin{align}\label{keyinequality}
u\Delta{w} \geqslant 
-2\alpha\nabla u \cdot \nabla w+\frac{2\alpha}{n}w^2+ K_1\alpha Aw+K_2\beta Bw+ I_1\alpha A^2+I_2B^2+I_3\beta AB,
\end{align}
where
\begin{equation}\label{eqI_iK_j}
\left\{
\begin{split}
&I_1:=\dfrac{2}{n}(1-2\alpha)^2-2\alpha^2+\alpha,\\
&I_2:=1+\dfrac{2}{n}\alpha\beta^2-\dfrac{q-1}{2}\beta^2,\\
&I_3:=\dfrac{q-1}{2}\Big(\dfrac{q+1}{2}-\alpha\Big)- \alpha \Big(q-\dfrac{8\alpha}{n}+\dfrac{4}{n}\Big),\\
&K_1:=1+\dfrac{4(1-2\alpha)}{n},\\
&K_2:=\dfrac{q-1}{2}-\dfrac{4\alpha}{n}.
\end{split}
\right.
\end{equation}
\end{lemma}

\begin{proof}
We write the Laplacian of $w$ as follows 
\begin{equation}\label{IJ}
\Delta w = -\Delta^2 u + \alpha \Delta(A) + \beta \Delta (B) = u^{-q} +\alpha \Delta(A) + \beta \Delta (B).
\end{equation}
We shall compute $\Delta (B)$ and find a lower bound of $\Delta (A)$. First, for the term $\Delta (B)$, a straightforward computation shows that
\begin{equation*}
\begin{split}
 \Delta (B)&= -pu^{-p-1} \Delta u + p(p+1) u^{-p-2} |\nabla u|^2\\
& = - p u^{-p-1}(-w + \alpha u^{-1} |\nabla u|^2 + \beta u^{-p})+p(p+1) u^{-p-2} |\nabla u|^2 .
\end{split}
\end{equation*}
Hence, 
\begin{equation}\label{kqJ}
u \Delta (B)=pBw + p(p+1-\alpha) AB - p\beta B^2.
\end{equation}
For the term $ \Delta (A)$, we have
\begin{align*}
\Delta (A)& = \frac{\Delta |\nabla u|^2}u -2\frac{\nabla (|\nabla u|^2)\cdot \nabla u}{u^2} + |\nabla u|^2 \Delta (u^{-1})\\
&= 2 \frac{\|\nabla^2 u\|^2}{u} + 2 \frac{\nabla u\cdot \nabla \Delta u} u - 4\frac{\nabla^2 u(\nabla u) \cdot \nabla u}{u^2} +|\nabla u|^2 \big(2u^{-3} |\nabla u|^2- u^{-2} \Delta u \big),
\end{align*}
where $\nabla^2 u$ denotes the Hessian matrix of $u$ and $\|\cdot\|$ denotes the Hilbert--Schmidt norm on matrices defined to be
\[
\| M \| = \big( \sum_{i,j} |m_{ij}|^2 \big)^{1/2},
\]
where the matrix $M$ is given as follows $M =(m_{ij})$. Using $\Delta u=-w+\alpha A+\beta B$, we have
\begin{align*}
\nabla u\cdot \nabla \Delta u =& \nabla u\cdot \nabla(-w + \alpha u^{-1} |\nabla u|^2 + \beta u^{-p})\\
=&- \nabla u \cdot \nabla w + 2\alpha u^{-1} \nabla^2 u(\nabla u) \cdot \nabla u -\alpha u^{-2} |\nabla u|^4 -p\beta u^{-p-1} |\nabla u|^2.
\end{align*}
Therefore, 
\begin{align*}
u\Delta (A)=&- 2\nabla u \cdot \nabla w + 2 \left\| \nabla^2 u + (\alpha-1) \frac{\nabla u\otimes \nabla u}u\right\|^2 -2\alpha(\alpha-1) u^{-2} |\nabla u|^4\\
& -2p\beta u^{-p-1} |\nabla u|^2 -|\nabla u|^2 u^{-1} \Delta u\\
=&- 2\nabla u \cdot \nabla w+ 2 \left\| \nabla^2 u + (\alpha-1) \frac{\nabla u\otimes \nabla u}u\right\|^2 -2\alpha(\alpha-1) A^2 \\
& -2p\beta AB-A \Delta u.
\end{align*}
For any $n\times n$ matrix $M$, it is elementary that $\|M\|^2 = {\rm Tr}(M^t M) \geqslant n^{-1} ({\rm Tr} (M))^2$, hence
\begin{align*}
 \left\| \nabla^2 u + (\alpha-1) u^{-1} \nabla u\otimes \nabla u \right\|^2 &\geqslant n^{-1} \big(\Delta u + (\alpha -1) u^{-1} |\nabla u|^2 \big)^2\\
 &=n^{-1} \big(-\Delta u + (1-\alpha) A \big)^2,
\end{align*}
which implies that
\begin{equation*}
u\Delta (A) \geqslant - 2 \nabla u \cdot \nabla w+ \frac{2}{n} \big( -\Delta{u} +(1-\alpha)A \big)^2-2\alpha(\alpha-1)A^2 
 -2p\beta AB +A (-\Delta u).
\end{equation*} 
Replacing $-\Delta u$ by $w-\alpha A -\beta B$ in the preceding estimate, we deduce that 
\begin{equation*}
\begin{split}
u\Delta (A)\geqslant &- 2 \nabla u \cdot \nabla w
 +\frac{2}{n}w^2+ \frac{4}{n}\Big((1-2\alpha)A-\beta B)\Big)w +\frac{2}{n}\Big((1-2\alpha)A-\beta B)\Big)^2\\
 &-2\alpha(\alpha-1)A^2 
 -2p\beta AB +A w- A(\alpha A +\beta B).
\end{split}
\end{equation*}
From this we obtain
\begin{equation}\label{kqI}
\begin{split}
u\Delta (A)\geqslant &- 2 \nabla u \cdot \nabla w
 + \frac{2}{n}w^2+\frac{4}{n}\Big((1-2\alpha)A-\beta B\Big)w+A w \\
 &+\Big(\frac{2}{n}(1-2\alpha)^2-2\alpha^2+\alpha\Big)A^2 \\
& -\Big(1+2p+\frac{4}{n}(1-2\alpha)\Big)\beta AB+\frac{2}{n}\beta^2 B^2.
\end{split}
\end{equation}
Combining \eqref{IJ}, \eqref{kqJ}, and \eqref{kqI} we arrive at
\begin{align*}
u\Delta{w} \geqslant & 
B^2- 2\alpha\nabla u \cdot \nabla w+ \frac{2\alpha}{n}w^2+ \frac{4\alpha}{n}\Big((1-2\alpha)A-\beta B)\Big)w+\alpha A w \\
 &+\alpha\Big(\frac{2}{n}(1-2\alpha)^2-2\alpha^2+\alpha\Big)A^2 
 -\Big(1+2p+\frac{4}{n}(1-2\alpha)\Big)\alpha\beta AB+\frac{2}{n}\alpha\beta^2 B^2\\
 &+p\beta Bw + p(p+1-\alpha)\beta AB - p\beta^2 B^2.
\end{align*}
By a simple computation, we get
\begin{equation}\label{est1}
\begin{split}
u\Delta{w} \geqslant &
- 2\alpha\nabla u \cdot \nabla w+\frac{2\alpha}{n}w^2+ \Big(\frac{4\alpha(1-2\alpha)}{n}+\alpha\Big)Aw+(p-\frac{4\alpha}{n})\beta Bw \\
 &+\alpha\Big(\frac{2}{n}(1-2\alpha)^2-2\alpha^2+\alpha\Big)A^2+\Big(1+\frac{2}{n}\alpha\beta^2-p\beta^2\Big)B^2 \\
 &+\Big(p(p+1-\alpha)- \alpha(1+2p-\frac{8\alpha}{n}+\frac{4}{n})\Big)\beta AB ;
\end{split}
\end{equation}
hence the lemma follows. 
\end{proof}

To make use of the maximum principle, we are forced to require that $I_1, I_2, I_3\geqslant 0$ and that $K_1, K_2>0$. We will eventually make these conditions precise. First, by requiring $I_1, I_2, I_3\geqslant 0$ (and taking into account $K_1, K_2>0$ eventually), the following lemma is a direct consequence of Lemma~\ref{lem1}.

\begin{lemma}\label{iterationwk}
	Let $u$ be a smooth positive solution of the equation \eqref{eqMain} in ${\mathbb R}^n$. 
	Assume \eqref{conditionI}, then the function $w$ defined in \eqref{def:w} satisfies in ${\mathbb R}^n$ the differential inequality
	\begin{align}\label{keyestimate}
	u\Delta{w} \geqslant 
	-2\alpha\nabla u \cdot \nabla w+\frac{2\alpha}{n}w^2+ K_1\alpha Aw+K_2\beta Bw,
	\end{align}
	where $K_1, K_2$ are defined in \eqref{eqI_iK_j}.
\end{lemma}

With Lemma~\ref{iterationwk} in hand, we are in position to prove Theorem~\ref{thmPointwise}.

\begin{proof}[Proof of Theorem~\ref{thmPointwise} completed] 
 We first note that the assumption \eqref{conditionI}  in Theorem~\ref{thmPointwise}  guarantees $I_1, I_2, I_3\geqslant 0$. 
	Let $\gamma\in [0,1)$ verifying \eqref{ConditionsForH2}. Let  $w_\gamma :=u^{-\gamma}w$, we have the following
\begin{align}\label{grad}
\nabla w=u^{\gamma}\nabla w_{\gamma}+\gamma u^{\gamma-1}w_\gamma\nabla u. 
\end{align}
Using Lemma \ref{iterationwk}, we compute lower estimate of $\Delta w_{\gamma}$ as follows.
\begin{align*}
\Delta w_\gamma&= u^{-\gamma}\Delta w +2\nabla(u^{-\gamma})\cdot\nabla w+ \Delta(u^{-\gamma})w\\
&\geqslant u^{-\gamma-1}\Big(-2\alpha\nabla u \cdot \nabla w+\frac{2\alpha}{n}w^2+ K_1\alpha Aw+K_2\beta Bw\Big) \\
&\quad +2\nabla(u^{-\gamma})\cdot\nabla w+ \Delta(u^{-\gamma})w\\
&=u^{-\gamma-1}\Big(-2(\alpha+\gamma) \nabla u \cdot \nabla w+ \frac{2\alpha}{n}w^2+ K_1\alpha Aw+K_2\beta Bw\Big) \\
&\quad + \big(-\gamma u^{-\gamma-1}\Delta u +\gamma(\gamma+1)u^{-\gamma-1}A\big)w.
\end{align*}
Substituting $\nabla w=u^{\gamma}\nabla w_{\gamma}+\gamma u^{\gamma-1}w_\gamma\nabla u$ and $-\Delta u=w-\alpha A-\beta B$, we get
\begin{align*}
\Delta w_\gamma\geqslant &-2(\alpha+\gamma)u^{-\gamma-1} \nabla u \cdot (u^{\gamma}\nabla w_{\gamma}+\gamma u^{\gamma-1}w_\gamma\nabla u)+ \frac{2\alpha}{n}u^{-\gamma-1}w^2 \\
& + u^{-\gamma-1}w\Big(K_1\alpha A+K_2\beta B + \gamma (w-\alpha A-\beta B) +\gamma(\gamma+1)A\Big).
\end{align*}  
Taking into account $w=u^{\gamma}w_\gamma$ we obtain
\begin{align}
u^{1-\gamma}\Delta w_\gamma\geqslant J_1w^2_\gamma  +u^{-\gamma} (-2J_2 \nabla u \cdot \nabla w_{\gamma} + L_1Aw_\gamma+L_2 Bw_\gamma) ,\label{28replace}
\end{align} 
where 
\begin{equation*}
\left\{
\begin{split}
J_1&:=\frac{2\alpha}{n}+\gamma,\\
J_2&:=\alpha+\gamma,\\
L_1&:=K_1\alpha-3\gamma\alpha-\gamma^2+\gamma,\\
L_2&:=(K_2-\gamma)\beta.
\end{split}
\right.
\end{equation*}
It is noted that the condition \eqref{ConditionsForH2} guarantees $L_1, L_2>0$. We shall show that $w_\gamma\leqslant 0$ by way of contradiction. Suppose that
\[
M=\sup_{{\mathbb R}^n} w_\gamma>0.
\] 
(Clearly $M\leqslant +\infty$.) We have the following two possible cases:

\medskip\noindent\textbf{Case 1}. {\it Suppose that $w_\gamma$ has the global maximum}. In this scenario, there exists some $x_0 \in {\mathbb R}^n$ such that
\[
w_\gamma(x_0)=M=\max\limits_{{\mathbb R}^n} w_\gamma>0.
\] 
Clearly at $x_0$, we have that $\nabla w_\gamma(x_0)=0$ and that $\Delta w_\gamma(x_0)\leqslant 0$. The inequality~\eqref{28replace} at $x=x_0$ gives
\begin{align*}
0 \geqslant 
J_1M^2,
\end{align*}
which is impossible.

\medskip\noindent\textbf{Case 2}. {\it Suppose that the supremum of $w_\gamma$ is attained at infinity}. 
Let $\psi : {\mathbb R}^N \to [0,1]$ be a smooth cuff-off function in ${\mathbb R}^N$ satisfying
\[\left\{
\begin{split}
\psi&=1 \quad \text{ in } B_{1/2} \text{ and }\\
\psi&=0 \quad \text{ if }|x|>1.
\end{split}
\right.\]
Let also $\varphi=\psi^2$. Then for some uniform constant $C>0$ we have the upper bounds
\begin{equation}\label{bound}
|\Delta \varphi|\leqslant C, \;\; \varphi^{-1}|\nabla \varphi|^2\leqslant C.
\end{equation}
For $R>1$, let $\varphi_R(x)=\varphi (x/R)$ and
$$w_{\gamma,R}=\varphi_R\, w_\gamma.$$
 Noting that $w_{\gamma,R}$ is zero if $|x|>R$, then there exists $x_R\in B_R$ such that
$$M_R:=\max\limits_{{\mathbb R}^n}w_{\gamma,R}= w_{\gamma,R}(x_R).$$ 
The choice of $\varphi$ implies that $M_R\to M$ as $R\to \infty$. 
In what follows, we shall derive a contradiction by passing $R\to \infty$. Since $M>0$, we may assume without loss of generality that $M_R>0$ for all $R>1$. The property of local maximum gives
\[
\left\{
\begin{split}
\nabla w_\gamma=& -\frac{\nabla \varphi_R}{\varphi_R}w_\gamma,\\
0\geqslant &\Delta w_{\gamma,R} \qquad \text{ at } x=x_R.
\end{split}
\right.
\]
Hence, at $x=x_R$, we have 
\[\begin{split}
0\geqslant &\varphi_R \Delta w_\gamma +2\nabla \varphi_R\cdot\nabla w_\gamma+ \Delta\varphi_Rw_\gamma\\
=&\varphi_R\; \Delta w_\gamma-2\varphi_R^{-1}|\nabla \varphi_R|^2 w_\gamma+ \Delta\varphi_R w_\gamma.
\end{split}\]
This combined with \eqref{bound} yields
\begin{equation}
CR^{-2}w_\gamma\geqslant \varphi_R\;\Delta w_\gamma
\end{equation}
at $x=x_R$. Next, combining this with inequality~\eqref{28replace}, at $x=x_R$ we have
\begin{align*}
Cu^{1-\gamma} R^{-2}w_\gamma &\geqslant J_1\varphi_Rw^2_\gamma +\varphi_Ru^{-\gamma} (-2J_2 \nabla u \cdot \nabla w_{\gamma} + L_1Aw_\gamma+L_2 Bw_\gamma)\\
&=J_1\varphi_Rw^2_\gamma +\varphi_Ru^{-\gamma} (2J_2 \varphi_R^{-1}\nabla u \cdot \nabla\varphi_R + L_1A+L_2 B)\, w_{\gamma}\\
&\geqslant J_1\varphi_Rw^2_\gamma+  \varphi_Ru^{-\gamma}\big(-\varepsilon^{-1}J_2^2 \varphi_R^{-2}|\nabla \varphi_R|^2u+ (L_1-\varepsilon) A+L_2 B \big)w_\gamma,
\end{align*}
where we have used the Cauchy--Schwarz inequality
\[
2J_2\varphi_R^{-1}\nabla u \cdot \nabla\varphi_R \geqslant -\varepsilon^{-1}J_2^{2} \varphi_R^{-2}|\nabla \varphi_R|^2u-\varepsilon  A 
\]
at $x=x_R$ once to get the latter estimate. Consequently,
\begin{align}\label{212bis}
C(\varepsilon)u^{1-\gamma} R^{-2}w_\gamma &\geqslant J_1\varphi_Rw^2_\gamma+  \varphi_Ru^{-\gamma}\big( (L_1-\varepsilon) A+L_2 B \big)w_\gamma
\end{align}
at $x=x_R$. Where $C(\varepsilon)$ is a postive constant depending on $\varepsilon$ but it does not depend on $R$. Choosing $\varepsilon>0$ small such that $L_1>\varepsilon$, keeping only the first term of the right hand side of \eqref{212bis}, we have at $x=x_R$
\begin{align}\label{endest}
C(\varepsilon) u^{1-\gamma} R^{-2}\geqslant J_1 \varphi_R w_\gamma=J_1M_R\to J_1M
\end{align}
as $R\to \infty$.
From the assumption \eqref{hypo} that  $u(x)=o\big( |x|^{2/(1-\gamma)} \big)$, we have a contradiction in \eqref{endest} when $R$ is sufficiently large.  Hence,  $ w_\gamma\leq 0$ and \eqref{eqPointwiseInequality} follows. 

 Finally, to treat the last part of the theorem, we first deduce from $q>1$ and the third inequality of \eqref{conditionI} that 
 \begin{equation}\label{condiK1}
 q-1-\frac{8\alpha}{n}>0.
 \end{equation}
 Next, the first inequality of \eqref{conditionI} guarantees
\begin{equation}\label{condiK2}
\alpha+\frac{4\alpha(1-2\alpha)}{n}>0.
\end{equation} 
It follows from \eqref{condiK1} and  \eqref{condiK2} that \eqref{ConditionsForH2} holds for sufficiently small $\gamma>0$. In this setting, the hypothesis \eqref{hypo} immediately implies that $u(x) = O(|x|^2)$. Theorem is proved.
\end{proof}


\section{Pointwise inequality for the system (\ref{system}): Proof of Theorem~\ref{th2}}

The section is devoted to a proof of Theorem~\ref{th2}. For simplicity, we denote
\[
w = l u^\sigma-v,
\]
where $\sigma = (1-q)/(r+1)$ and $l = (-\sigma)^{-1/(r+1)}$. Hence to conclude the theorem, it suffices to establish $w \leqslant 0$. Note that $\sigma <0$, a direct calculation leads us to
\[\begin{split}
\Delta w =& l\sigma u^{\sigma-1}\Delta u + l \sigma(\sigma-1)|\nabla u|^2-\Delta v\\
 \geqslant & l\sigma u^{\sigma-1}\Delta u -\Delta v\\
= & l\sigma u^{\sigma-1}v^r +u^{-q}\\
=&-l\sigma u^{\sigma-1}(l^ru^{\sigma r}-v^r).
\end{split}
\]
Replacing $u^{\sigma}= l^{-1}(w+v)$ we have, 
\begin{equation*}
\Delta w\geqslant -l\sigma \big(l^{-1}(w+v)\big)^{(\sigma-1)/\sigma}\big( (w+v)^r-v^r\big)
\end{equation*}
Or we can write
\begin{equation}\label{deltaw}
\Delta w\geqslant C(w+v)^{(\sigma-1)/\sigma}\big( (w+v)^r-v^r\big)
\end{equation}
We now prove $w \leqslant 0$ by way of contradiction. Suppose that 
\[
M=\sup_{{\mathbb R}^n} w(x) > 0 .
\] 
(Clearly $M\leqslant +\infty$.) There are two possible cases:

\medskip\noindent\textbf{Case 1}. If $x_0$ is a global maximum point of $w$ in ${\mathbb R}^n$, that is $w(x_0) = \sup_{x \in {\mathbb R}^n} w(x)$ with $\Delta w(x_0) \leqslant 0$. Clearly this cannot be the case since $w(x_0) > 0$ implies  from \eqref{deltaw} that $\Delta w(x_0) > 0$.

\medskip\noindent\textbf{Case 2}. The supremum of $w$ is attained at infinity. 
Let $\psi: {\mathbb R}^N \to [0,1]$ be a smooth cuff-off function in ${\mathbb R}^N$ such that $\psi=1$ in $B_{1/2}$ and $\psi=0$ if $|x|>1$. Let also $\varphi=\psi^m$ for some $m>0$ to be chosen later. Then as always, there is some uniform constant $C>0$ such that
\begin{equation}\label{bound1x}
|\Delta \varphi|\leqslant C\varphi^{1-2/m}, \quad  \varphi^{-1}|\nabla \varphi|^2\leqslant C\varphi^{1-2/m}.
\end{equation}
For each $R>1$ we set
\[ 
\varphi_R(x)=\varphi(x/R), \quad
w_R = \varphi_R w.
\]
Noting that $w_R$ is zero if $|x|>R$, then there exists $x_R\in B_R$ such that
$$M_R:=\max\limits_{{\mathbb R}^n} w_R= w_R(x_R).$$ 
The choice of $\varphi$ implies that $M_R\to M$ as $R\to \infty$.  This implies that $M_R>0$ for $R$ large. We may assume that $M_R>0$ for all $R>1$. Next, the property of local maximum gives
\[
\left\{
\begin{split}
\nabla w=&-\frac{\nabla \varphi_R}{\varphi_R}\,w,\\
0\geqslant & \Delta w_R \qquad \text{ at } x=x_R.
\end{split}
\right.
\]
Hence, at $x=x_R$, we have 
\[
\begin{split}
0\geqslant &\varphi_R\,\Delta w +2\nabla \varphi_R\cdot\nabla w+\Delta\varphi_R\,w\\
 =&\varphi_R\; \Delta w-2\varphi_R^{-1}|\nabla \varphi_R|^2 w+ \Delta\varphi_R w.
\end{split}
\]
 This combined with \eqref{bound1x} yields
\begin{equation}\label{sds}
 CR^{-2}\varphi_R^{1-2/m}w\geqslant \varphi_R\;\Delta w
\end{equation}
at $x=x_R$. We now consider two cases of $r$. If $r\geqslant 1$ then 
using the inequality $(a+b)^r- a^r\geqslant b^r$ for $a,b\geqslant 0$, it follows from \eqref{deltaw} and \eqref{sds} that, at $x=x_R$ 
\begin{align*}
 CR^{-2}\varphi_R^{1-2/m}w&\geqslant \varphi_R\;(w+v)^{(\sigma-1)/\sigma}\big( (w+v)^r-v^r\big)\\
 &\geqslant \varphi_R\;(w+v)^{(\sigma-1)/\sigma} w^r\\
 &\geqslant \varphi_R\;w^{(\sigma-1)/\sigma+r} .
\end{align*}
Hence, 
\begin{align*}
 CR^{-2}\geqslant \varphi_R^{2/m}(x_R) \;w^{(\sigma-1)/\sigma+r-1}(x_R).
\end{align*}
Choosing $m>0$ such that $2/m=(\sigma-1)/\sigma+r-1$, we arrive at
\begin{align*}
 CR^{-2}\geqslant \varphi_R^{2/m}(x_R) \;w^{2/m}(x_R)=M_R^{2/m},
\end{align*}
which is impossible when $R$ is large. If $r<1$ then the concavity of function $f(s)=s^r$ in $(0, \infty)$ implies that 
\begin{equation}\label{mean}
(w+v)^r-v^r \geqslant rw (v+ w)^{r-1}
\end{equation}
at $x=x_R$. We deduce from \eqref{deltaw}, \eqref{sds} and \eqref{mean} that, at $x=x_R$, 
\begin{align*}
CR^{-2}\varphi_R^{1-2/m}w&\geqslant \varphi_R\;(w+v)^{(\sigma-1)/\sigma}\big( (w+v)^r-v^r\big)\\
&\geqslant \varphi_R\;(w+v)^{(\sigma-1)/\sigma} rw (v+ w)^{r-1}\\
&=r\varphi_R\;(w+v)^{(\sigma-1)/\sigma+r-1}\; w\\
&\geqslant r\varphi_R\;w^{(\sigma-1)/\sigma+r}. 
\end{align*}
Hence, 
\begin{align*}
CR^{-2}\geqslant \varphi_R^{2/m}(x_R) \;w^{(\sigma-1)/\sigma+r-1}(x_R).
\end{align*}
Again, choosing $m>0$ such that $2/m=(\sigma-1)/\sigma+r-1$, we arrive at
\begin{align*}
CR^{-2}\geqslant \varphi_R^{2/m}(x_R) \;w^{2/m}(x_R)=M_R^{2/m},
\end{align*}
which is impossible when $R$ is large. The proof is complete.


\section{Pointwise inequality for the parabolic system (\ref{parasystem}): Proof of Theorem~\ref{th16}}

In this section, we prove Theorem~\ref{th16}. Denote
\[
w = u-l v^\sigma,
\]
where $\sigma := (r+1)/(p+1)$ and $l :=\sigma^{-1/(p+1)}$.  By a direct computation and taking into account $\sigma \in (0,1]$, we have
\begin{align}
\Delta w &\geqslant \Delta u- l\sigma v^{\sigma-1}\Delta v\notag \\
&=u_t-v^r+ l\sigma v^{\sigma-1}(u^p-v_t)\notag\\
&=l\sigma v^{\sigma-1}(u^{p}-l^pv^{\sigma p})+w_t.\label{5.1}
\end{align}
We  prove $w \leqslant 0$ by way of contradiction. Suppose that 
\[
M=\sup\limits_{{\mathbb R}^n\times \mathbb R} w(x,t) > 0.
\] 
(Clearly $M\leqslant +\infty$.) There are two possible cases:

\medskip\noindent\textbf{Case 1}. If $(x_0,t_0)$ is a global maximum point of $w$ in ${\mathbb R}^n\times \mathbb R$, that is 
$$w(x_0,t_0) = \sup_{(x,t) \in {\mathbb R}^n\times \mathbb R} w(x)=M>0$$
 with $\Delta w(x_0,t_0) \leqslant 0$ and $w_t(x_0,t_0) = 0$. We have a contradiction with \eqref{5.1} at $(x_0,t_0)$.

\medskip\noindent\textbf{Case 2}. The supremum of $w$ is attained as $|x|+|t|^{1/2}\to \infty$. 
Let $\psi : {\mathbb R}^N \to [0,1]$ be a smooth cuff-off function in ${\mathbb R}^N$ such that $\psi=1$ if $ |x|+|t|^{1/2}\leqslant 1/2$ and $\psi=0$ if $|x|+|t|^{1/2}>1$. Let also $\varphi=\psi^m$ for some $m>0$ to be chosen later. Then we have the upper bound
\begin{equation}\label{bound1}
|\varphi_t|+ |\Delta \varphi|+ \varphi^{-1}|\nabla \varphi|^2\leqslant C\varphi^{1-2/m}.
\end{equation}
For each $R>0$ we set
\[ 
\varphi_R(x,t)=\varphi(x/R, t/R^2), \quad
w_R = \varphi_R w.
\]
Since $w_R$ is zero if $|x|+|t|^{1/2}>R$, then there exists $(x_R, t_R)$ such that
$$M_R:=\max\limits_{{\mathbb R}^n\times \mathbb R} w_R= w_R(x_R, t_R).$$ 
The choice of $\varphi$ implies that $M_R\to M$ as $R\to \infty$. The property of local maximum gives
\[
\left\{
\begin{split}
w_t=&-\frac{(\varphi_R)_t}{\varphi_R}\,w,\\
\nabla w=&-\frac{\nabla \varphi_R}{\varphi_R}\,w,\\
0\geqslant & \Delta w_R \qquad \text{ at } (x_R,t_R).
\end{split}
\right.
\]
Combining this with \eqref{5.1}, at $(x_R, t_R)$, we have 
\[
\begin{split}
0\geqslant &\varphi_R\,\Delta w +2\nabla \varphi_R\cdot\nabla w+\Delta\varphi_R\,w\\
\geq&\varphi_R\Big(l\sigma v^{\sigma-1}(u^{p}-l^pv^{\sigma p})+w_t\Big)+2\nabla \varphi_R\cdot\nabla w+ \Delta\varphi_R w\\
=& \varphi_R\,l\sigma v^{\sigma-1}(u^{p}-l^pv^{\sigma p}) - (\varphi_R)_t w-2\varphi_R^{-1}|\nabla \varphi_R|^2 w+ \Delta\varphi_R w
\end{split}
\]
This combined with \eqref{bound1} yields
\begin{equation}\label{5.3}
CR^{-2}\varphi_R^{1-2/m}w\geqslant \varphi_R\,l\sigma v^{\sigma-1}(u^{p}-l^pv^{\sigma p})
\end{equation}
at $(x_R, t_R)$. We now consider following two subcases.

\medskip\noindent\textit{Case 2.1}. If the sequence $v(x_R, t_R)$ is bounded, then  $v^{\sigma-1}(x_R, t_R)$ is bounded below away from zero. Recalling $p>1$, then using the elementary inequality 
\begin{equation}\label{triangle}
(a+b)^p- a^p\geqslant b^p, \quad a,b\geqslant 0,
\end{equation}
 it follows from \eqref{5.3}  that, at $(x_R, t_R)$ 
\begin{align*}
CR^{-2}\varphi_R^{1-2/m}w&\geqslant C\varphi_R\;w^p.
\end{align*}
Hence, 
\begin{align*}
CR^{-2}\geqslant \varphi_R^{2/m}(x_R,t_R) \;w^{p-1}(x_R,t_R).
\end{align*}
Choosing $m>0$ such that $2/m=p-1$, we arrive at
\begin{align*}
CR^{-2}\geqslant \varphi_R^{p-1}(x_R,t_R) \;w^{p-1}(x_R,t_R)=M_R^{p-1},
\end{align*}
which is impossible when $R$ is large. 

\medskip\noindent\textit{Case 2.2}. If the sequence $v(x_R,t_R)$ is unbounded, then, up to a subsequence, we may assume that $$\lim\limits_{R\to\infty} v(x_R,t_R)=\infty.$$
Since $pr>1$ and $p>1$, there exists  $\varepsilon>0$ small enough such that $pr-1-\varepsilon(r+1)>0$ and $p>1+\varepsilon$. For $0<b<a$, using the convexity of function $f(s)=s^{p/(1+\varepsilon)}$ in $(0, \infty)$ and \eqref{triangle},  we have
\begin{align*}
a^p-b^p&=(a^{1+\varepsilon})^{p/(1+\varepsilon)}-(b^{1+\varepsilon})^{p/(1+\varepsilon)}\\
&\geqslant \frac{p}{1+\varepsilon}(b^{1+\varepsilon})^{p/(1+\varepsilon)-1}(a^{1+\varepsilon}-b^{1+\varepsilon}) \\ 
&= \frac{p}{1+\varepsilon}(b^{p-\varepsilon-1})(a^{1+\varepsilon}-b^{1+\varepsilon})\\
&\geqslant  \frac{p}{1+\varepsilon}(b^{p-\varepsilon-1})(a-b)^{1+\varepsilon}.
\end{align*}
Taking $a=u(x_R, t_R)$, $b=lv^{\sigma}(x_R, t_R)$, we deduce from \eqref{5.3} that, at $(x_R, t_R)$,  
\begin{align*}
CR^{-2}\varphi_R^{1-2/m}w&\geqslant \varphi_R\, v^{\sigma-1}v^{\sigma(p-\varepsilon-1)}w^{1+\varepsilon}\\
&=\varphi_R v^{\frac{pr-1-\varepsilon(r+1)}{p+1}}w^{1+\varepsilon}\\
&\geqslant  C\varphi_R w^{1+\varepsilon}.
\end{align*}
where $C>0$ independent of $R$,  and in the last inequality we  used the unboundedness of the sequence $v(x_R,t_R).$
Hence, 
\begin{align*}
CR^{-2}\geqslant \varphi_R^{2/m}(x_R,t_R) \;w^{\varepsilon}(x_R,t_R).
\end{align*}
Again, choosing $m>0$ such that $2/m=\varepsilon$, we arrive at
\begin{align*}
CR^{-2}\geqslant \varphi_R^{\varepsilon}(x_R,t_R) \;w^{\varepsilon}(x_R,t_R)=M_R^{\varepsilon},
\end{align*}
which is impossible when $R$ is large. Theorem is proved.


\section*{Acknowledgments}
 The authors are deeply grateful to Professor Dong Ye for his valuable comments and suggestions on the preminilary version of this work. This work was done while QAN and QHP were visiting the Vietnam Institute for Advanced Study in Mathematics (VIASM). They gratefully acknowledges the institute for its hospitality and support. The research of QAN is funded by Vietnam National Foundation for Science and Technology Development (NAFOSTED) under grant number 101.02-2016.02. VHN was supported by CIMI postdoctoral research fellowship.



\end{document}